\documentclass[12pt]{amsart}
\usepackage{amsfonts,amssymb,amscd,amsmath,enumerate,verbatim}
\usepackage[latin1]{inputenc}
\usepackage{amscd}
\usepackage{latexsym}
\usepackage{mathptmx}
\usepackage{multicol}
\input xy
\xyoption{all}

\def\NZQ{\mathbb}               
\def\NN{{\NZQ N}}

\def\frk{\mathfrak}               

\def\Phi{{\frk N}}
\def\G{\Gamma}

\def\opn#1#2{\def#1{\operatorname{#2}}} 
\opn\chara{char} 
\opn\length{\ell} 
\opn\pd{pd} 
\opn\rk{rk}
\opn\projdim{proj\,dim} 
\opn\injdim{inj\,dim} 
\opn\rank{rank}
\opn\depth{depth} 
\opn\grade{grade} 
\opn\height{height}
\opn\embdim{emb\,dim} 
\opn\codim{codim}

\opn\Tr{Tr} 
\opn\bigrank{big\,rank}
\opn\superheight{superheight}
\opn\lcm{lcm}
\opn\trdeg{tr\,deg}
\opn\G{\mathcal{G}}
\opn\reg{reg} 
\opn\lreg{lreg} 
\opn\ini{in} 
\opn\lpd{lpd}
\opn\size{size}
\opn\mult{mult}
\opn\dist{dist}
\opn\cone{cone}
\opn\lex{lex}
\opn\rev{rev}
\opn\im{im}
\opn\m{m}
\opn\div{div} \opn\Div{Div} \opn\cl{cl} \opn\Cl{Cl}
\opn\Spec{Spec} \opn\Supp{Supp} \opn\supp{supp} \opn\Sing{Sing}
\opn\Ass{Ass} \opn\Min{Min}
\opn\Ann{Ann} \opn\Rad{Rad} \opn\Soc{Soc}
\opn\Syz{Syz} \opn\Im{Im} \opn\Ker{Ker} \opn\Coker{Coker}
\opn\Am{Am} \opn\Hom{Hom} \opn\Tor{Tor} \opn\Ext{Ext}
\opn\End{End} \opn\Aut{Aut} \opn\id{id} \opn\ini{in}

\opn\nat{nat}
\opn\pff{pf}
\opn\Pf{Pf} \opn\GL{GL} \opn\SL{SL} \opn\mod{mod} \opn\ord{ord}
\opn\Gin{Gin}
\opn\Hilb{Hilb}\opn\adeg{adeg}\opn\std{std}\opn\ip{infpt}
\opn\Pol{Pol}
\opn\sat{sat}
\opn\Var{Var}
\opn\Gen{Gen}
\opn\ch{char}

\opn\aff{aff} \opn\con{conv} \opn\relint{relint} \opn\st{st}
\opn\lk{lk} \opn\cn{cn} \opn\core{core} \opn\vol{vol}
\opn\link{link} \opn\star{star}
\opn\gr{gr}

\def\pot#1#2{#1[\kern-0.28ex[#2]\kern-0.28ex]}

\opn\dirlim{\underrightarrow{\lim}}
\opn\inivlim{\underleftarrow{\lim}}
\let\to=\rightarrow

\def\Implies{\ifmmode\Longrightarrow \else
        \unskip${}\Longrightarrow{}$\ignorespaces\fi}
\def\implies{\ifmmode\Rightarrow \else
        \unskip${}\Rightarrow{}$\ignorespaces\fi}
\def\iff{\ifmmode\Longleftrightarrow \else
        \unskip${}\Longleftrightarrow{}$\ignorespaces\fi}

\let\:=\colon

\newtheorem{Theorem}{Theorem}[section]
\newtheorem{Lemma}[Theorem]{Lemma}

\newtheorem{Definition}[Theorem]{Definition}

\newtheorem{Conjecture}[Theorem]{Conjecture}

\let\epsilon\varepsilon
\let\phi=\varphi
\let\kappa=\varkappa
\def\qed{\ifhmode\textqed\fi
      \ifmmode\ifinner\quad\qedsymbol\else\dispqed\fi\fi}
\def\textqed{\unskip\nobreak\penalty50
       \hskip2em\hbox{}\nobreak\hfil\qedsymbol
       \parfillskip=0pt \finalhyphendemerits=0}
\def\dispqed{\rlap{\qquad\qedsymbol}}

\opn\dis{dis}
\opn\height{height}
\opn\dist{dist}
\def\pnt{{\raise0.5mm\hbox{\large\bf.}}}

\opn\Lex{Lex}

\begin{document}
	
\title{Relation between regularity of powers of edge ideals and $(\im, \reg)$-invariant extension}
\author{Hiroju Kanno}

\address{Hiroju Kanno, 
	Department of Pure and Applied Mathematics,
	Graduate School of Information Science and Technology,
	Osaka University, Suita, Osaka 565-0871, Japan}
\email{u825139b@ecs.osaka-u.ac.jp}

\begin{abstract}
	In this paper, we define $(\im, \reg)$-invariant extension of graphs and propose a new approach for Nevo and Peeva's conjecture which said that for any gap-free graph $G$ with $\reg(I(G)) = 3$ and for any $k \geq 2$, $I(G)^k$ has a linear resolution. 
	Moreover, we consider new conjectures related to the regularity of powers of edge ideals of gap-free graphs. 
\end{abstract}
\maketitle
\section{INTRODUCTION}

Let $R = K[x_1, \ldots, x_n]$ denote the polynomial ring in $n$ variables over a field $K$ with each $\deg x_i = 1$ and $M$ is a finitely generated graded $S$-module. 
The graded minimal free resolution of $M$ is of the form 
\[
\mathbb{F}_{M} : 0 \to \bigoplus_{j \ge 0} R(-j)^{\beta_{p, j}(M)} \to \cdots \to \bigoplus_{j \ge 0} R(-j)^{\beta_{0, j}(M)} \to M \to 0. 
\]

Suppose the graded minimal free resolution of $M$ is the above form. $\beta_{i, j}(M)$ is called $\{i, j\}$-th graded Betti number of $M$ and we denote the ({\em Castelnuovo--Mumford}\,) {\em regularity} by $\reg(M)$, 
the {\em projective dimension} of $M$ by $\pd(M)$. 	
\[\reg(M) = \{j-i : \beta_{i, j}(M) \neq 0\}. \]
\[\pd(M) = \{i : \beta_{i, j}(M) \neq 0\} = p. \]
If there exists the integer $d$ such that $\beta_{i, j}(M) = 0$ for all $j \neq i + d$ (i.e., $\reg(M) = d$), we call $M$ has a {\em linear resolution}. 
  
Let $G = (V(G), E(G))$ be a finite simple graph on the vertex set $V(G) = \{x_{1}, \ldots, x_{n}\}$ with the edge set $E(G)$.  Let $R = K[V(G)] = K[x_{1}, \ldots, x_{n}]$. The {\em edge ideal} of $G$ is the ideal 
\[
I(G) = \left( x_{i}x_{j} : \{x_{i}, x_{j}\} \in E(G) \right) \subset R. 
\] 

The study of the relationship between $\reg(I(G))$ and combinatorial invariant of $G$ is a central problem in combinatorial commutative algebra. In particular, the study of the regularity of the powers of edge ideals is 
a current trend. See \cite{B1}, \cite{BHT}, \cite{E2}, \cite{HHZ}, \cite{NP}. 

In this work, we define $(\im, \reg)$-{\em invariant extension} of graphs, 
obtain the partially result for the famous conjecture \ref{NPconj} and propose several conjectures.

\begin{Conjecture}\cite{NP}\label{NPconj}
	Let $G$ be a gap-free graph and $I(G)$ its edge ideal. Suppose $\reg(I(G)) = 3$, is it true that $I(G)^s$ has a linear resolution for all $s \geq 2$? 
\end{Conjecture}

The following is our main result. 

\begin{Theorem}(Theorem\ref{main2})
	Let $G$ be a gap-free graph and $S \subset V(G)$ be an independent set of $G$, and $G^S$ be an $S$-suspension of $G$. If $I(G)^k$ has a linear resolution for all $k \geq 2$, 
	then $I(G^S)^k$ has a linear resolution for all $k \geq 2$.  
\end{Theorem}

We explain why our result relates Conjecture \ref{NPconj} directly. 
For this, we define $(\im, \reg)$-invariant extension of graphs which is a generalization of $S$-suspension. 

\begin{Definition}\upshape
	Let $G=(V(G), E(G))$ and $G'=(V(G'), E(G'))$ be graphs. 
	If $G'$ satisfies the following two conditions, we call $G'$ an $(\im, \reg)$-{\em invariant extension} of $G$.
	\begin{enumerate}
		\item $V(G') = V(G) \cup \{x_{n+1}\}$ where $x_{n+1}$ is a new vertex, and $G'_{V(G)}$ is $G$. 
		\item $\im(G') = \im(G)$ and $\reg(I(G)) = \reg(I(G'))$. 
	\end{enumerate}
\end{Definition}

The following new Conjecture \ref{Newconj1} of $(\im, \reg)$-invariant extension follows Conjecture \ref{NPconj}. (See Theorem \ref{Theoconj})

\begin{Conjecture}\label{Newconj1}
	Let $G$ and $G'$ be gap-free graphs such that $G'$ is $(\im, \reg)$-invariant extension of $G$. 
	If $\reg(I(G)^k) = 2k$ for all $k \geq 2$, 
	is it true that $\reg(I(G')^k) = 2k$ for all $k \geq 2$ ?
\end{Conjecture}

\begin{Theorem}\label{Theoconj}
	Conjecture \ref{Newconj1} implies Conjecture \ref{NPconj}. 
\end{Theorem}

\begin{proof}
	For any gap-free graph $G$ such that $\reg(I(G)) = 3$, there exists 
	the integer $n_G \geq 5$ such that an anticycle of length $n$ is an induced subgraph of $G$ and $G$ has no induced anticycle of length less than $n_G$. Hence $|V(G)| \geq n_G$. We denote $|V(G)| - n_G$ by $m_G$.
	
	We prove by induction on $m_G$. 
	If $m_G = 0$, then $G$ is an anticycle of length $n_G$. 
	For all $l \geq 5$, an anticycle of length $l$ is gap-free and cricket-free.
	Then $I(G)^k$ has a linear resolution for all $k \geq 2$. 
	Suppose $m_G > 1$, there exists a graph $H$ such that 
	$H$ is an induced subgraph of $G$, $|V(H)| = |V(G)| - 1$, and $\reg(I(H)) = 3$. 
	By the assumption of the induction, $I(H)^k$ has a linear resolution for all $k \geq 2$.
	Therefore $I(G)^k$ has a linear resolution for all $k \geq 2$ by Conjecture \ref{Newconj1}. 
\end{proof}

\section{Preliminaries}
In this section, we fix notation and recall some known results. 
Firstly, we introduce the terms of graph theory. 
Throughout this paper, a graph $G$ denotes a finite simple graph with vertex set $V(G)$ and edge set $E(G)$, which has no isolated vertex. 

\begin{Definition}\upshape
	Let $G=(V(G), E(G))$ be a graph. 
	\begin{enumerate}
		\item A graph $G'$ is a subgraph of $G$ if $V(G) \subset V(G)$ and $E(G') \subset E(G)$. 
		\item A graph $G'$ is an {\em induced subgraph} of $G$ if $G'$ is a subgraph of $G$ and for any $x, y \in V(G')$, 
		$\{x, y\} \in E(G') \iff \{x, y\} \in E(G)$. 
		\item A graph $G^c$ is the complement of $G$ if $V(G)=V(G^c)$ and $\{x, y\} \in E(G^c) \iff \{x, y\} \notin E(G)$. 
	\end{enumerate}
\end{Definition}

We denote by $G_{W}$ the induced subgraph of $G$ such that $V(G_{W}) = W \subset V(G)$. 

\begin{Definition}\upshape
	Let $G=(V(G), E(G))$ be a graph. 
	\begin{enumerate}
		\item A subset $M = \{e_{1}, \ldots, e_{s} \} \subset E(G)$ is said to be a {\em matching} of $G$ if, for all $e_{i}$ and $e_{j}$ with $i \neq j$ belonging to $M$, one has $e_i \cap e_j = \emptyset$.  The {\em matching number} $\m(G)$ of $G$ is the maximum cardinality of the matchings of $G$. 
		\item A matching $M = \{e_{1}, \ldots, e_{s} \} \subset E(G)$ is said to be an {\em induced matching} of $G$ if for all $e_{i}$ and $e_{j}$ with $i \neq j$ belonging to $M$, there is no edge $e \in E(G)$ with $e \cap e_{i} \neq \emptyset$ and $e \cap e_{j} \neq \emptyset$.   
		The {\em induced matching number} $\im(G)$ of $G$ is 
			the maximum cardinality of the induced matchings of $G$. 
		\item A subset $S \subset V(G)$ is said to be an {\em independent set} of $G$ if $\{ x_{i}, x_{j} \} \not\in E(G)$ for all $x_{i}, x_{j} \in S$.  
	\end{enumerate}
\end{Definition}

Two disjoint edges $e_1, e_2$ in $G$ are said to be {\em gap} if for any edge $e \in E(G)$, $e \cap e_1 = \emptyset$ or $e \cap e_2 = \emptyset$.
 If $G$ has no gap, $G$ is called {\em gap-free}.
 $G$ is gap-free if and only if $\im(G) = 1$. 

\begin{Definition}\upshape
	Let $G=(V(G), E(G))$ be a graph. 
	\begin{enumerate}
		\item $G$ is called {\em cycle} if $V(G) = \{x_1, \ldots , x_n\}$ and 
		$E(G) = \{\{x_i, x_{i+1}\} : 1 \le i \le n, x_{1} = x_{n+1}\}$. Then $n$ is called the length of cycle.  
		\item $G$ is called {\em anticycle} if the complement of $G$ is cycle. 
		\item $G$ is called {\em chordal} graph if $G$ has no induced cycle of length $4$ or more. 
		\item $G$ is called {\em co-chordal} graph if $G^c$ is chordal graph. 
		\item $G$ is called {\em claw} graph if $V(G) = \{x_1, x_2, x_3, x_4\}, E(G) = \{\{x_1, x_2\}, \{x_1, x_3\}, \{x_1, x_4\}\}$. A graph without any induced claw is called a {\em claw-free} graph.
		\item $G$ is called {\em cricket} graph if $V(G) = \{x_1, x_2, x_3, x_4, x_5\}$, 
		
		$E(G) = \{\{x_1, x_2\}, \{x_1, x_3\}, \{x_1, x_4\}, \{x_1, x_5\}, \{x_4, x_5\}\}$.
		A graph without any induced cricket is called a {\em cricket-free} graph. 
	\end{enumerate}
\end{Definition}

If $G$ is a cricket-free graph then $G$ is a claw-free graph. 

\begin{Definition}\upshape
	Let $G=(V(G), E(G))$ be a graph and $x \in V(G)$. 
	\begin{enumerate}
		\item A subset $N_G(x) \subset V(G)$ is {\em neighborhood} of $x$ if 
		$N_G(x) = \{y \in V(G) : \{x, y\} \in E(G)\}$. 
		\item The degree of $x$ is $\deg_G(x) = |N_G(x)|$. 
	\end{enumerate}
\end{Definition}

Secondly, we note related known results. 

\begin{Theorem}\cite{F}
	Let $G$ be a graph. The followings are equivalent: 
	\begin{itemize}
		\item $I(G)$ has a linear resolution. In other words, $\reg(I(G)) = 2$. 
		\item The complement of $G$ is a chordal graph.  
	\end{itemize}
\end{Theorem}

This is the classical result by Fr\"oberg in \cite{F}. 
The following result is generalization of the above by Herzog, Hibi and Zheng in \cite{HHZ}. 

\begin{Theorem}\cite{HHZ}
	Let $G$ be a graph. The followings are equivalent: 
	\begin{itemize}
		\item $I(G)$ has a linear resolution. In other words, $\reg(I(G)) = 2$. 
		\item The complement of $G$ is a chordal graph.  
		\item Each power of $I$ has a linear resolution.
		In other words, for any $k \geq 1$, $\reg\left(I(G)^k\right) = 2k$. 
		\item $I(G)$ has a linear quotient. 
	\end{itemize}
\end{Theorem}

If $G^c$ is a chordal graph, $G^c$ has no induced cycle of length $4$ or more, thus $\im(G) = 1$. 
In general, it is known that the regularity of edge ideal of $G$ is bounded below by the induced matching number of $G$, and above by the matching number. 

\begin{Theorem}\cite{HaVanTuyl}, \cite{K}
	For any graph $G$, the following holds.  
	\[
	\im(G) + 1 \leq \reg(I(G)) \leq \m(G) + 1.
	\] 
\end{Theorem}

Moreover, the regularity of powers of edge ideals is bounded below by the $\im(G)$. 

\begin{Theorem}\cite{BHT}
	For any graph $G$ and for all $k \geq 1$, the following holds.
	\[2k + \im(G) - 1 \leq \reg(I(G)^k). \]
\end{Theorem}

The above implies that if there exists the integer $k \geq 1$ such that $I(G)^k$ has a linear resolution, then $\im(G) = 1$. Conversely, if $\im(G) = 1$, is there the integer $k \geq 1$ such that $I(G)^k$ has a linear resolution? It is an open problem. 
Conjecture \ref{NPconj} is the special and important case. 
We know several results for Conjecture \ref{NPconj}, in particular, the following is well-known.

\begin{Theorem}\cite{B1}
	For any gap-free and cricket-free graph $G$ and for all $k \geq 2$, 
	$\reg(I(G)) \leq 3$ and 
	$\reg(I(G)^k) = 2k$, as a consequence, $I(G)^k$ has a linear minimal free resolution. 
\end{Theorem}

Finally, we introduce the $S$-suspension and prepare several lemmata to prove Theorem \ref{main2}.

\begin{Definition}\upshape\cite{HNM}
	Let $G=(V(G), E(G))$ be a graph and $S \subset V(G)$ be an independent set of $G$. 
	The $S$-{\em suspension} of $G$, denoted by $G^S$, is defined as follows:
	\begin{enumerate}
		\item $V(G^{S}) = V(G) \cup \{ x_{n + 1} \}$, where $x_{n + 1}$ is a new vertex. 
		\item $E(G^{S}) = E(G) \cup \left\{ \{ x_{i}, x_{n + 1} \} : x_{i} \not\in S \right\}$. 
	\end{enumerate}
\end{Definition}

\begin{Lemma}
	\label{S-cc}\cite{HNM}
	Let $G=(V(G) = \{x_{1}, \ldots, x_{n}\}, E(G))$ be a graph and $S\subset V(G)$ be an independent set of $G$. 
	Then one has 
	\begin{enumerate}
		\item[$(1)$]  $\im(G^{S}) = \im(G)$. 
		\item[$(2)$]  $\reg(I(G^{S})) = \reg(I(G))$. 
	\end{enumerate}
\end{Lemma}

\begin{Definition}\upshape\cite{FHV}
	Let $I, J,$ and $K$ be monomial ideals such that $\G(I)$ is the disjoint union of $\G(J)$ and $\G(K)$ where $\G(I)$ is the set of the minimal monomial generators of $I$.
	$I = J+K$ is a {\em Betti splitting} if 
	\[\beta_{i,j}(I) = \beta_{i,j}(J)+\beta_{i,j}(K)+\beta_{i-1,j}(J \cap K) 
	\text{for all $i \in \NN$ and (multi)degrees $j$.}\]
	If $I = J+K$ is a Betti splitting then the followings are hold. 
	\begin{enumerate}
		\item $\reg(I) = \max\{\reg(J), \reg(K), \reg(J\cap K) -1\}$. 
		\item $\pd(I) = \max\{\pd(J), \pd(K), \pd(J\cap K) + 1\}$. 
	\end{enumerate}
\end{Definition}

\begin{Lemma}\cite{FHV}\label{doublelinear}
	Let $I$ be a monomial ideal in $R$, and suppose that $J$ and $K$ are monomial ideals in 
	$R$ such that $\G(I)$ is the disjoint union of $\G(J)$ and $\G(K)$.
	If both $J$ and $K$ have linear resolutions, then $I = J + K$ is a Betti splitting.
\end{Lemma}

\begin{Lemma}\label{exact}
	Let $I \subseteq R$ be a monomial ideal, and let $m$ be a monomial of degree $d$.
	Then
	\[
	\reg(I) \leq \max\{ \reg (I : m) + d, \reg (I,m)\}. 
	\]
	Moreover, if $m$ is a variable $x$ appearing in $I$, then $\reg(I)$ is {\it equal} to one of
	these terms.
\end{Lemma}

\begin{Lemma}\cite{B1}
	Let $I$ and $J$ be monomial ideals in $R$ generated in degrees $n_1$ and $n_2$ respectively.
	Assume $J \subset I$ and $n_1 < n_2$ and $\G(I) = \{m_1,m_2,...,m_k\}$. 
	We put
	\[
	\begin{array}{rcl}
	A & = & \max \{\reg (J:m_1)+n_1\}, \\
	B & = & \max \{\reg ((J,m_1,..,m_l):m_{l+1})+n_1|1\leq l \leq {k-1}\}, \\
	C & = & \reg(I).
	\end{array}
	\]
	Then $\reg(J) \leq \max \{A,B,C\}$. 
\end{Lemma}

\begin{Theorem}\cite{B1}\label{Blemma}
	Let $G$ be a graph and $I=I(G)$, and $\G(I^n) = (L_1^{(n)}, \ldots , L_m^{(n)})$. 
	For all $1 \leq k$ and for all $j \leq k \leq m - 1$, if $(L_j ^{(n)}: L_{k+1} ^{(n)} )$  is not contained in $(I^{n+1}:L_{k+1} ^ {(n)})$ then there exists $i \leq k$ such that $(L_i ^{(n)}: L_{k+1} ^{(n)})$ is generated by a variable and $(L_j ^{(n)} : L_{k+1} ^{(n)}) \subseteq (L_i ^{(n)} : L_{k+1} ^ {(n)} )$. 
\end{Theorem}

\section{Main result}

In this section, we prove our main result Theorem \ref{main2}. This is implied by Theorem \ref{main1}. 

\begin{Lemma}\label{keylemma}
	Let $G$ be a graph and $U$ be a vertex cover of $G$. Suppose 
	$\G(I(G)^k) = \{L_{k_1}, \ldots, L_{k_l}\}$. Then 
	$(UI(G)^k : L_{k_i})$ is generated by variables for all $k \geq 0$ and $k_l \geq i \geq 1$. 
\end{Lemma}

\begin{proof}
	This proof is based on the argument of Theorem 6.1 in \cite{B1}.
	We prove by induction on $k$. 
	If $k = 0$, then it is obviously. 
	For all $k \geq 1$ and $k_l \geq i \geq 1$, $L_i$ can be written as \[L_i = e_1e_2 \cdots e_k,  \text{\, where\, } e_j \in E(G).\]
	Let $m$ be a minimal monomial generator of $(UI(G)^k : L_{k_i})$
	Now $me_1e_2 \cdots e_k$ is divisible by a monomial $f = uf_1f_2 \cdots f_k$, where $f_j \in E(G)$ and $u \in U$. 
	If $\deg(m) \geq 2$, there are $i$ and $e_i = pq$ such that 
	$f_1f_2 \cdots f_k|me_1 \cdots e_{i-1}qe_{i+1} \cdots e_k$.
	
	Without loss of generality, we may assume $e_i = e_1$.
	If $q = u$, then $f_1 \cdots f_k|me_2 \cdots e_k$. 
	If $q|f_1f_2 \cdots f_k$, we have $f_1 = p'q$. 
	It follows that $uf_2 \cdots f_k$ divides $me_2 \cdots e_k$. 
	If $q \nmid uf_1 \cdots f_k$, then $uf_1 \cdots f_k|me_2 \cdots e_k$. 
	Therefore $uf_2 \cdots f_k$ divides $me_2 \cdots e_k$. 
	Hence for each case, $m$ is a variable by the assumption of induction. 
\end{proof}

\begin{Theorem}\label{main1}
	Let $G$ be a gap-free graph and $S \subset V(G)$ be an independent set of $G$, 
	and $G^S$ be an $S$-suspension of $G$. If $I(G)^k$ has a linear resolution for all $k \geq 2$, 
	then $I(G^S)^k = I(G)^k + JI(G^S)^{k-1}$ is a Betti splitting, where $J = (x_0x_i : x_i \notin S)$. 
\end{Theorem}

\begin{proof}
	It is obvious that $\G(I(G^S)^k)$ is a disjoint union of $\G(I(G)^k)$ and $\G(JI(G^S)^{k-1})$. 
	By Lemma \ref{doublelinear}, we may only show that $JI(G^S)^{k-1}$ has a linear resolution. 
	By Lemma \ref{exact} and $(JI(G^S)^{k-1}:x_0) = (x_i : x_i \notin S)I(G^S)^{k-1}$, 
	$JI(G^S)^{k-1}$ has a linear resolution if and only if $(x_i : x_i \notin S)I(G^S)^{k-1}$ has a linear resolution. 
	We denote $(x_i : x_i \notin S)I(G^S)^{k-1}$ by $T$. 
	Suppose $\G(I(G)^{k-1}) = \{L_1, \ldots , L_m\}$. 
	Because $S$ is an independent set of $G$, $\{x_i : x_i \notin S\} \subset V(G^S)$ is a vertex cover of $G^S$. 
	Hence $I(G)^{k} \subset T$.  
	By Theorem \ref{Blemma}, \[\reg(T) \leq \max \{\reg((T:L_j) + 2(k - 1), \reg(I(G)^{k-1})\}.\]
	According to Theorem \ref{keylemma}, for all $1 \leq j \leq m$, $(T:L_j)$ is generated by variables. 
	It implies that \[\reg(T) \leq \max \{2k-1, 2(k-1)\}.\]
	Therefore $T$ has a linear resolution. 
\end{proof}

\begin{Theorem}\label{main2}
	Let $G$ be a gap-free graph and $S \subset V(G)$ be an independent set of $G$, 
	and $G^S$ be an $S$-suspension of $G$.
	If $I(G)^k$ has a linear resolution for all $k \geq 2$, 
	then $I(G^S)^k$ has a linear resolution for all $k \geq 2$. 
\end{Theorem}

\begin{proof}
	By \ref{main1}, if $I(G)^k \cap JI(G^S)^{k-1}$ has a linear resolution, $I(G^S)^k$ has a linear resolution. 
	Because $\{x_i : x_i \notin S\} \subset V(G^S)$ is a vertex cover of $G^S$, 
	we have $I(G)^k \cap JI(G^S)^{k-1} = x_0I(G)^k$. 
	By Lemma \ref{exact}, $x_0I(G)^k$ has a linear resolution. 
\end{proof}

\section{Conjectures}

The following is known as a generalization of Conjecture \ref{NPconj}. 

\begin{Conjecture}\label{generalNP}
	Let $G$ be a gap-free graph and $I(G)$ its edge ideal. Suppose $\reg(I(G)) = n$.
	Is it true that $I(G)^s$ has a linear resolution for all $s \geq n - 1$? 
\end{Conjecture}

We can consider two conjectures for Conjecture \ref{generalNP} as well as \ref{NPconj}. 

\begin{Conjecture}\label{Newconj2}
	Let $G=(V(G), E(G))$ be a gap-free graph and $G'$ be an $(\im, \reg)$-invariant extension of $G$. 
	If there exists an integer $c_G$ such that $\reg\left(I(G)^k\right) = 2k$ for all $k \geq c_G$, 
	is it true that $\reg\left(I(G')^k\right) = 2k$ for all $k \geq c_G$ ?
\end{Conjecture}

\begin{Conjecture}\label{Newconj3}
	For all integers $n \geq 2$, is there the family of graphs $\{G^{(n)}_i\}$ such that 
	$I(G^{(n)}_i)^k$ has a linear resolution for all $k \geq n - 1$ and 
	if $G$ is a gap-free graph such that $\reg(I(G)) = n$, 
	then $G^{(n)}_i$ is an induced subgraph of $G$. 
\end{Conjecture}

If Conjecture \ref{Newconj2} and Conjecture \ref{Newconj3} hold, Conjecture \ref{generalNP} holds too.
The proof is similar to that of Theorem \ref{Theoconj}.

\end{document}